\documentclass[reqno, 12pt]{amsart}
\usepackage{amsmath, amsthm, amssymb, mathtools, graphicx, paralist, cite, color}
\usepackage{comment}

\usepackage{hyperref}

 \usepackage{amsthm}
\usepackage[utf8]{inputenc}
\usepackage{mathtools}
\usepackage{amssymb}
\usepackage{amsmath} 
\usepackage{amsthm}

\newtheorem{theorem}{Theorem}

\newtheorem{lemma}{Lemma}

\newtheorem{corollary}{Corollary}%[lemma]
\theoremstyle{remark}
\newtheorem{remark}{Remark}
\usepackage{indentfirst}
\usepackage{listings}
\usepackage{hyphenat}
\usepackage{physics}
\usepackage{calrsfs}
\usepackage{enumerate}
\usepackage{tikz-cd}
\usepackage{mathrsfs}
\usepackage{bbm}
\usepackage{comment}
\usepackage{url}

\title{An $\Omega$-Result for the Counting of Geodesic Segments in the Hyperbolic Plane}
\author{Marios Voskou}
\address{Max Planck Institute for Mathematics, Vivatsgasse 7, 53111 Bonn, Germany}
\email{voskou@mpim-bonn.mpg.de}
\date{\today}
\subjclass[2020]{Primary 11F72}

\begin{document}
\begin{abstract} 
Let $\Gamma$ be a cocompact Fuchsian group, and $l$ a fixed closed geodesic. We study the counting of those images of $l$ under the action of $\Gamma$ that have a distance from $l$ less than or equal to $R$. We prove an $\Omega$-result for the error term in the asymptotic expansion of the counting function. More specifically, we prove that, for every $\delta>0$, the error term is equal to $\Omega_{\delta}\left(X^{1/2}\left(\log{\log{X}}\right)^{1/4-\delta} \right)$, where $X=\cosh{R}$.
\end{abstract}
\maketitle
\section{Introduction}
\subsection{Classical Lattice Counting Problem}
Let $\Gamma$ be a cofinite Fuchsian group, and $\mathbb{H}=\left.\left\{ z=x+iy \in \mathbb{C} \right\vert y>0 \right\}$ be the hyperbolic upper half-plane. For fixed $z,w \in \mathbb{H}$, the classical hyperbolic lattice counting problem is concerned with estimating the number of elements in the orbit of $z$ that are within distance $R$ from $w$. In other words, for $X=2\cosh{R}$, we are interested in estimating the quantity
$$N(X,z,w):= \# \left\{ \gamma \in \Gamma \vert \;  2\cosh{\rho(\gamma z, w) } \leq X \right\},$$
as $X$ tends to $+\infty$.
%We note that  $$N(X,z,w)= \sum_{\gamma \in \Gamma} \mathbbm{1}_{\left[0,X \right]}\left( 2\cosh{\rho(\gamma z, w) }  \right),$$
%i.e., $N(X,z,w)$ is the automorphization of $\mathbbm{1}_{\left[0,X \right]}\left( 2\cosh{\rho(\gamma z, w) }\right)$, where $\mathbbm{1}_A(u)$ is the indicator function of the set $A$.

Using the spectral theory of automorphic kernels, Selberg \cite{selberg} proved that
\begin{equation*} N(X,z,w)=\sum_{1/2<s_j \leq 1} \sqrt{\pi}\frac{\Gamma \left(s_j-\frac{1}{2} \right)}{\Gamma\left(s_j+1 \right)} u_{j}(z) \overline{u_{j}(w)}X^{s_j}+E(X,z,w), \end{equation*}
%where the main term $M(X,z,w)$ is given by
%$$M(X,z,w)=\sum_{1/2<s_j \leq 1} \sqrt{\pi}\frac{\Gamma \left(s_j-\frac{1}{2} \right)}{\Gamma\left(s_j+1 \right)} u_{j}(z) \overline{u_{j}(w)}X^{s_j} \sim \frac{\pi}{\hbox{Vol}(\Gamma \backslash \mathbb{H})} X,$$
%and
where the error term satisfies
$E(X,z,w)=O(X^{2/3}).$ %The implied constant depends on $z,w$ and the group $\Gamma$.
Here, the sum is over the small eigenvalues $\lambda_j=s_j(1-s_j)<1/4$ of the Laplacian of the hyperbolic surface $\Gamma \backslash \mathbb{H}$. Furthermore, $(u_j)_j$ is a corresponding maximal $L^2$-orthonormal set of eigenfunctions.

It is conjectured that, for any $\epsilon>0$, \begin{equation} %\label{conj1} 
E(X,z,w)=O_{\epsilon}(X^{1/2+\epsilon}). \end{equation}
No improvement of the exponent $2/3$ has been achieved so far for any group $\Gamma$ and for any points $z,w$. However, it is  known that, if $\Gamma$ is cocompact and $1/4$ is not an eigenvalue, then $E(X,z,z)\neq O(X^{1/2}).$
%, implying that equation (\ref{conj1}) is the best bound we can hope for. 
This is due to a result of Philips--Rudnick \cite{rudnick}. They define $e(X,z,w)$ to be the error term without the contribution %to be theFor a certain quantity $a(X) \asymp X^{1/2}\log{X}$ depending only on $X$,
%they define $$e(X,z,w):=E(X,z,w)-a(X) \cdot \sum_{\lambda_j=1/4} u_{j}(z) \overline{u_{j}(w)}, $$
%i.e., the error term without the contribution
of $\lambda_j=1/4$, and prove the following result. %Here, $h$ is the Selberg/Harish\hyp{}Chandra transform of the indicator function of $[0,(X-2)/4]$. 
%They prove the following result. 
\begin{theorem}[{\cite[Theorem 1.2]{rudnick}}] \label{omegaold}
If $\Gamma$ cocompact, then, for any $\delta>0$, we have
$$e(X,z,z)=\Omega_{\delta}\left(X^{1/2}\left(\log{\log{X}}\right)^{1/4-\delta} \right).$$
%In particular,
%$$e(X,z,z) \neq O\left(X^{1/2} \right).$$
\end{theorem}
They also prove weaker results for a general cofinite group $\Gamma$%, showing that, for any $\epsilon>0$, $e(X,z,z)$ cannot be equal to $O_{\epsilon}\left(X^{1/2-\epsilon} \right)$
. In \cite{chatzakosomega}, Chatzakos generalized this for some cases with $z \neq w$. 
\begin{comment}
\begin{proof}
For any $x \in \mathbb{R}$, we define $\left\langle x \right\rangle$ as the distance of $x$  from the nearest integer. We have that
$$\left| e^{2 \pi xi}-1\right|^2=(\cos{2 \pi x}-1)^2+\sin^2{(2 \pi x)}=4\sin^2{(\pi\left\langle x \right\rangle)}  \leq 4\pi^2 \left\langle x \right\rangle^2,$$
i.e.,
$$\left| e^{2 \pi x i}-1\right| \leq 2 \pi  \left\langle x \right\rangle.$$
Therefore, it is enough to find $R$ such that, for every $j$, we have $\left\langle r_jR/2\pi\right\rangle< 2/T$. We partition $[0,1)$ into $\lfloor T \rfloor$ intervals of equal length.
 By Dirichlet's box principle, there are two distinct non-negative integers $k_1, k_2 \leq T^n$  such that $\left\langle r_j(k_1-k_2)M/2\pi\right\rangle<1 /\lfloor T \rfloor <2/T$ for every $j$. Let $R=|k_1-k_2| \cdot M$. For every $j$, we have
 $$\left|e^{r_jRi}-1 \right| <  2 \pi  \left\langle r_jR/2\pi \right\rangle < \frac{4 \pi}{T},$$
as required.
\end{proof}
\end{comment}

\subsection{Counting Hyperbolic-Hyperbolic Double Cosets}
In this paper, we prove an $\Omega$-result for another, more complicated hyperbolic counting problem. In particular, let $l$ be a closed geodesic segment with stabilizer $\Gamma_1=\langle \gamma_1 \rangle$, where $\gamma_1$ is a primitive hyperbolic element of $\Gamma$. We count the number of images $\gamma l$ of $l$ under $\Gamma_1 \backslash \Gamma \slash \Gamma_1$ that are of distance at most $R$ from $l$. Various authors consider this problem, including Martin--Mckee--Wambach \cite{mckee}, Tsuzuki \cite{tsuzukiletter}, Parkkonen--Paulin \cite{parkkonen2}, Good \cite{good} and Lekkas \cite{lekkas}. A refined version has also been studied by Good \cite{good}, Hejhal \cite{hejhal}, and Voskou \cite{voskou}.

Up to conjugation, we can assume that $\gamma_1$ is diagonal and, hence, %In particular, we will have that
%$$\gamma_1=\begin{pmatrix} m & 0 \\ 0 & m^{-1}\end{pmatrix}$$
%for some $m>1$, 
that $l$ lies on the positive imaginary axis $I$. We denote its hyperbolic length by $\hbox{len}(l)$. 
In \cite[Lemma 1]{mckee}, Martin--Mckee--Wambach show that,
 for $\gamma= \begin{pmatrix}
a & b\\
c & d
\end{pmatrix},$
we have $$\cosh{\hbox{dist}(\gamma l, l) }= \mathrm{max}\left(\left| B(\gamma) \right|,1\right), \,
\hbox{where } \; \; B(\gamma):=ad+bc.$$
Hence, for $X=\cosh{R}>1$, the given counting problem is equivalent to estimating
\begin{equation} \label{lekkasn} N(X,l):=\#\left\{ \gamma \in \Gamma_1 \backslash \Gamma \slash \Gamma_1 \vert \, \left| B(\gamma) \right| \leq X \right\}. \end{equation}

Let now $$M(X,l):=\frac{2\left(\mathrm{len}(l)\right)^2}{\pi\mathrm{Vol}\left(\mathbb{H}/\Gamma\right)}X+\sum_{1/2 < s_j < 1}D(s_j)\left|\hat{u}_{j}\right|^2X^{s_j},$$

where $\hat{u}_{j}$ denotes the hyperbolic period associated with $u_j(z)$, defined by
$$\hat{u}_{j}:=\int_{l} u_j(z) ds(z),$$
and $D(s)$ is defined by %($D(s)=\gamma_2(1-s)/\pi$, where $\gamma_2$ is as in ) (see (...)). 
%where 
$$D(s):=\frac{\Gamma \left(\frac{2s-1}{2} \right)\Gamma \left(\frac{s+1}{2} \right)}{\left(\Gamma \left(\frac{s}{2} \right)\right)^2\Gamma \left(\frac{s+2}{2} \right)}. $$

Good \cite{good} and Lekkas \cite{lekkas} prove independently the following asymptotic formula.

\begin{theorem} \label{lekkasmain} Let $\Gamma$ be a cocompact Fuchsian group. We have that
$$N(X,l)=M(X,l)+E(X,l),$$
where $E(X,l)$ satisfies the bound
 $E(X,l)=O\left( X^{2/3} \right).$

\end{theorem}

As in the classical hyperbolic lattice counting problem, it is conjectured that, for any $\epsilon>0$, we have \begin{equation} \label{conj2} E(X,l)=O_{\epsilon}(X^{1/2+\epsilon}).
\end{equation}
This is only known to be true on average. In particular, Lekkas \cite{lekkas} proved that
%\begin{theorem} \label{meanvaluelekkas}
  %  For $E(X,l)$ as in Theorem \ref{lekkasmain}, we have that
    $$ \frac{1}{X}\int_X^{2X}|E(x,l)|^2dx \ll X\log^2X.$$
%\end{theorem}
%\begin{remark}
The main ingredient in his proof is the large sieve inequality we develop in
%developed in
\cite{voskoulekkas}.
%\end{remark}
%\begin{remark}
    For corresponding results in the classical hyperbolic lattice counting problem, see \cite{chamizo2}. For a refined version of this result, see \cite{voskou}.
%\end{remark}

On the other hand, similarly with the classical hyperbolic lattice counting problem, we can prove an $\Omega$-result for the error term, implying that the upper bound (\ref{conj2}) is the best possible bound in the general case.

Similarly with \cite{rudnick}, let \begin{equation} \label{omtild}\tilde{E}(X,l)=E(X,l)-a_{1/2}(X,l)=N(X,l)-M(X,l)-a_{1/2}(X,l), \end{equation}
where $$a_{1/2}(X,l):=\frac{4}{\pi}\left(\sum_{s_j=1/2}\left|\hat{u}_{j}\right|^2\right) \cdot X \cdot {}_{3}F_{2}\left(1, \frac{1}{4}, \frac{1}{4}; \; \frac{3}{2}, \frac{1}{2} ; \, -X^2\right) \asymp X^{1/2}\log{X}$$ is the contribution of the eigenvalue $\lambda_j=1/4$, and ${}_{3}F_{2}$ denotes the generalized hypergeometric function. We will prove the following Theorem.
\begin{theorem} \label{omegathm}
For $\Gamma$ a cocompact Fuchsian group, and $\tilde{E}(X,l)$ as in equation (\ref{omtild}), we have
    $$ \tilde{E}(X,l)=\Omega_{\delta}\left(X^{1/2}(\log{\log{X}})^{1/4-\delta} \right).$$
\end{theorem}
%\begin{remark}
%    We can also find weaker $\Omega$-results when $\Gamma$ is a general cofinite group.
%\end{remark}
\begin{remark}
    For $\Omega$-results regarding the counting of images of a fixed point $z$ that have distance $\leq R$ from $l$, see the work of Chatzakos in \cite{chatzakosomega2}.
\end{remark}
\section{Preliminaries}

%One of the main methods in the study of such problems is the use of the spectral theory of automorphic forms. In particular, 
Let $\Delta$ be the Laplacian of the hyperbolic surface $\Gamma \backslash \mathbb{H}$ and let $(u_j)_j$ be a maximal $L^2$-orthonormal set of eigenfunctions (\emph{Maa{\ss} forms}) for the discrete spectrum of $\Delta$, with corresponding eigenvalues $(\lambda_j)_j$. We note that the eigenvalues are all real, and we write $\lambda_j=s_j(1-s_j)$, with $s_j$ either being of the form $1/2+it_j$ with $t_j \in \mathbb{R}$ , or a real number in the interval $\left(1/2,1\right]$. These correspond to the cases $\lambda_j \geq 1/4$ (in which case we call the eigenvalue \emph{large}) and $\lambda_j<1/4$ (in which case we call the eigenvalue \emph{small}), respectively.

In \cite{lekkas}, Lekkas uses the following relative trace formula, in order to relate $N(X,l)$, and therefore $E(X,l)$, to a spectral sum. In particular, the main term $M(X,l)$ is due to the contribution of the small eigenvalues, and the error term $E(X,l)$ is due to the contribution of the large eigenvalues.
%The counting function $N(X,l)$, and hence the error term $E(X,l)$, can be related to a spectral sum via the following relative trace formula, due to Lekkas
%\cite{lekkas}.

\begin{theorem}[\cite{lekkas}]
\label{lekkasthm}
 Let $f$ be a real, continuous, piecewise differentiable function with exponential decay. Let $\varepsilon$ be equal to $1$ if $\Gamma$ has an element with zero diagonal entries, and $0$ otherwise.
For $\Gamma$ a cocompact Fuchsian group, we have
$$ (1+\varepsilon)f(1)\mathrm{len}\left(l\right)+\sum_{\substack{\gamma \in \Gamma_1 \backslash \Gamma \slash \Gamma_1 \\ 
\gamma \neq \mathrm{id}}} g\left(B^2(\gamma)\right) =2\sum_jd_{t_j}\left( f \right)\left|\hat{u}_{j}\right|^2, $$
where $B(\gamma):=ad+bc$, and $$g(u)=
\int_{ \sqrt{\max\left(0,u-1 \right)}}^{\infty}  \frac{f\left( x^2+1  \right)}{\sqrt{x^2-\left(u-1 \right)}}dx,$$
and $d_t(f)$ is the Huber transform of $f$, given by
$$d_t(f):=\int_{0}^{\infty}f(x^2+1){}_2F_1\left(\frac{s}{2},\frac{1-s}{2};\frac{1}{2};-x^2\right) dx,$$
where $s=1/2+it$. Here, ${}_2F_1$ denotes the hypergeometric function.
\end{theorem}

\begin{remark}
    For a generalization of Lekkas' relative trace formula that holds for every cofinite Fuchsian group $\Gamma$, as well as a version that is sensitive to the sign of $B(\gamma)$, see \cite[Thm.7]{voskou}.
\end{remark}

Theorem \ref{lekkasthm} is a crucial ingredient in our proof as well. In particular, we will use this to write a mollified version of $\tilde{E}(X,l)$ as a spectral sum, with the terms being of the form $c_j e^{it_jR}$. We will then use combinatorics to prove that there exist arbitarily large $X$'s such that a chunk of the terms resonates. In particular, we will use the following combinatorial lemma, as in the proof of Theorem \ref{omegaold} (see \cite{rudnick}).

\begin{lemma} \label{phole}
    Let $r_1,r_2, \dots, r_n$ be $n$ distinct real numbers. Let $T>1$ be a real number. Then, for every $M>0$, there is some $R$ satisfying $M \leq R \leq MT^n$, such that $$\left|e^{ir_jR}-1 \right| < \frac{4 \pi}{T}$$ for all $j=1,2,\dots,n$.
\end{lemma}

%Here, the main term corresponds to the contribution of the small eigenvalues of $\Delta$ over $L^2\left(\Gamma \backslash \mathbb{H} \right)$ (i.e. $\lambda_j <1/4$), while the error term corresponds to the contribution of the large eigenvalues (i.e. $\lambda_j \geq 1/4$).
Finally, we will show that the contribution of the rest of the terms is insignificant, while the contribution of the resonated terms is enough to prove our $\Omega$-result (Theorem \ref{omegathm}). We will deal with the contribution of the periods using the following asymptotic formula.
%Regarding the periods $\hat{u}_{j}$, we have the following asymptotic formula.
\begin{lemma}[{\cite[Thm.1 p.2388]{tsuzuki}\label{periodslemma}}] If the group $\Gamma$ is cocompact, then, we have
\begin{equation*} 
\sum_{t_j \leq X}{ \left| \hat{u}_{j} \right|^2} \sim \frac{\mathrm{len}(l)}{\pi} X.
\end{equation*}
\end{lemma}
Via partial summation, we have the following useful corollary. %which will be important in our work.
\begin{corollary}
    \label{weightedperiods}
    If $\Gamma$ is cocompact, then, for any fixed real constant $a \neq -1$, we have that 
    $$\sum_{ t_j \leq T}{ t_j^{a}\left| \hat{u}_{j} \right|^2} \ll \max\left(T^{1+a},1\right).$$
    Furthermore, for $a<-1$, we also have that
    $$\sum_{ t_j > T}{ t_j^{a}\left| \hat{u}_{j} \right|^2} \ll T^{1+a}.$$
\end{corollary}

%The key ingredient in Lekkas' proof is the following relative trace formula.

%The following combinatorial lemma is a key ingredient in their proofs. 

\section{Proof of Theorem \ref{omegathm}}

Similarly with the proof of Theorem \ref{omegaold} in \cite{rudnick}, we use the mollification method. We consider the mollified normalized error term, \begin{equation} \label{omegainttr} e_{\epsilon}(R,l)=\int_{-\infty}^{+\infty}\frac{\psi_{\epsilon}(R-Y)\tilde{E}(\cosh Y,l)}{\sqrt{|\sinh{Y}|}} dY, \end{equation}
where $X=\cosh{R}$, $\epsilon \in (0,1)$ and $\psi_{\epsilon}$ is defined as follows. Let $\psi(x)$ be a smooth, even, non-negative function that is compactly
supported in $\left[-1,1\right]$, and such that
$$\hat{\psi}(t):=\int_{-\infty}^{+\infty} \psi(x)e^{-itx} dx  \geq 0$$
and $\hat{\psi}(0)=1$.
Then, for any $\epsilon>0$, we write $$\psi_{\epsilon}(x)=\epsilon^{-1}\psi(x/ \epsilon).$$
We note that $\psi_{\epsilon}(x)$ is compactly supported in $[-\epsilon,\epsilon]$, and that its Fourier transform satisfies $\hat{\psi}_{\epsilon}(t)=\hat{\psi}(\epsilon t).$
In particular, we have  $0 \leq \hat{\psi}_{\epsilon}(t) \leq 1$ and $\hat{\psi}_{\epsilon}(0)=1$. Furthermore, by integration by parts, for any fixed $k$, we have that \begin{equation} \label{fourierbound}\hat{\psi}_{\epsilon}(t)=O_k\left((\epsilon|t|)^{-k} \right). \end{equation}
We note that, as $\psi_{\epsilon}$ is compactly supported, equation (\ref{omegainttr}) gives
 $$e_{\epsilon}(R,l) \ll \frac{\tilde{E}(\cosh{R},l)}{\sqrt{\sinh{R}}}.$$
 Therefore, it is sufficient to show that
$$e_{\epsilon}(R,l)=\Omega_{\delta}((\log{R})^{1/4-\delta}).$$
Ideally, we would like to write $\tilde{E}$ as a spectral sum via the relative trace formula (Theorem \ref{lekkasthm}), using the indicator as the test function, and then apply the integral transformation given by equation (\ref{omegainttr}). Unfortunately, the indicator is not a valid choice for a test function, as it is discontinuous. %Therefore,
Instead, we first apply the transformation, and then use the relative trace formula on the mollified indicator. This makes the argument more technical.

By substituting equation (\ref{omtild}) into (\ref{omegainttr}), we write
$$e_{\epsilon}(R,l)=N_{\epsilon}(X,l)-M_{\epsilon}(X,l)-a_{1/2,\epsilon}(X,l),$$
where \begin{align*}N_{\epsilon}(X,l)&=\int_{-\infty}^{+\infty}\frac{\psi_{\epsilon}(R-Y)N(\cosh Y,l)}{\sqrt{\sinh{Y}}} dY, \\
M_{\epsilon}(X,l)&=\int_{-\infty}^{+\infty}\frac{\psi_{\epsilon}(R-Y)M(\cosh Y, l)}{\sqrt{\sinh{Y}}} dY,\end{align*}
and
\begin{equation} \label{defaheps} a_{1/2,\epsilon}(X,l)=\int_{-\infty}^{+\infty}\frac{\psi_{\epsilon}(R-Y)a_{1/2}(\cosh Y,l)}{\sqrt{\sinh{Y}}} dY. \end{equation}

In order to apply the relative trace formula (Theorem \ref{lekkasthm}), we write
\begin{align*}N_{\epsilon}(X,l)&=\int_{-\infty}^{+\infty}\frac{\psi_{\epsilon}(R-Y)}{\sqrt{\sinh{Y}}} \sum_{|B(\gamma)|<\cosh{Y}} 1 \;dY  \\
&= \int_{-\infty}^{+\infty}\frac{\psi_{\epsilon}(R-Y)}{\sqrt{\sinh{Y}}} \sum_{|B(\gamma)|<\cosh{R}} \mathbbm{1}(\left\{ |B(\gamma)|<\cosh{Y}\right\}) \;dY  \\
    &=\sum_{|B(\gamma)|<\cosh{R}} \int_{-\infty}^{+\infty}\frac{\psi_{\epsilon}(R-Y)}{\sqrt{\sinh{Y}}}  \mathbbm{1}(\left\{ |B(\gamma)|<\cosh{Y}\right\}) \;dY  \\
    &=\sum_{|B(\gamma)|<\cosh{R}} g_{\epsilon}(B^2(\gamma)), 
\end{align*} 
where $g_{\epsilon}(y)$ is the continuous function defined in $[0, +\infty)$ as being constant in $[0,1]$ and satisfying the relation
$$g_{\epsilon}(\cosh^2{x})=\int_{x}^{+\infty} \frac{\psi_{\epsilon}(R-Y)}{\sqrt{\sinh{Y}}}dY,$$
for any positive $x$.
By Theorem \ref{lekkasthm}, we deduce that
$$N_{\epsilon}(X,l)=2\sum_jd_{t_j}\left( f_{\epsilon} \right)\left|\hat{u}_{j}\right|^2+O(1),$$
where $f_{\epsilon}$ is a function such that, for $y \geq 1$,
$$g_{\epsilon}(y)=\int_{ y}^{+\infty}  \frac{f_{\epsilon}\left( t  \right)}{\sqrt{t-y}}\frac{dt}{\sqrt{t-1}}.$$
In particular, using the inversion formula for Weyl integrals from \cite[Eq. 1.64,1.62]{iwaniec}, we have
\begin{equation} \label{feps} f_{\epsilon}(u)=\frac{-\sqrt{u-1}}{\pi} \int_{u}^{+\infty} \frac{g_{\epsilon}'(y)}{\sqrt{y-u}} \, dy=\frac{-\sqrt{u-1}}{\pi} \int_{1}^{+\infty} \frac{g_{\epsilon}'(y)\mathbbm{1}(y>u)}{\sqrt{y-u}} \, dy. \end{equation}
Hence, using Euler's integral transformation for the hypergeometric function (see \cite[Eq.7.512.12, p.814]{toisap}), we have
\begin{align*}d_t(f_{\epsilon})&=\int_{0}^{+\infty}f_{\epsilon}(x^2+1) {}_2F_1\left(\frac{s}{2},\frac{1-s}{2};\frac{1}{2};-x^2 \right) dx  \\
&=-\frac{1}{\pi} \int_{1}^{+\infty}g'_{\epsilon}(y)\int_{0}^{\sqrt{y-1}}\frac{x}{\sqrt{y-1-x^2}} \cdot {}_2F_1\left(\frac{s}{2},\frac{1-s}{2};\frac{1}{2};-x^2 \right) dx dy  \\
&=-\frac{1}{\pi} \int_{-\infty}^{+\infty}g'_{\epsilon}(y)\int_{0}^{1}\frac{1}{2}\sqrt{\frac{y-1}{1-v}} \cdot {}_2F_1\left(\frac{s}{2},\frac{1-s}{2};\frac{1}{2};-(y-1)v \right) dv dy  \\
&= -\frac{1}{\pi} \int_{1}^{+\infty}g'_{\epsilon}(y)F_s(y) dy, \end{align*}
where
\begin{equation} \label{fsdef} F_{s}(y):=\sqrt{y-1} \cdot {}_3F_2\left(1,\frac{s}{2},\frac{1-s}{2};\frac{3}{2},\frac{1}{2};1-y\right). \end{equation}
Using the fact that
$$\sinh{2x} \cdot g'_{\epsilon}(\cosh^{2}(x))=-\frac{\psi_{\epsilon}(R-x)}{\sqrt{\sinh{x}}},$$
and noting that $\psi_{\epsilon}$ is compactly supported in $[-\epsilon,\epsilon]$,
we conclude that
$$d_t(f_{\epsilon})= \frac{1}{\pi} \int_{R-\epsilon}^{R+\epsilon}\frac{1}{\sqrt{\sinh{x}}}F_s(\cosh^2{x}) \psi_{\epsilon}(R-x) dx. $$
We now apply the transformation given in \cite[16.8.8]{dlmf} on (\ref{fsdef}) to write $F_s$ in terms of generalized hypergeometric functions with variables close to $0$, rather than infinity.
For $s \neq 0, 1/2,1,$ and any real number $A$, we have
\begin{align*}
 F_s(A^2+1)  &=A \cdot {}_3F_2\left(1,\frac{s}{2},\frac{1-s}{2};\frac{3}{2},\frac{1}{2};-A^2\right)  \\
&=\gamma_1(s){}_3F_2\left(1,\frac{3}{2},\frac{1}{2};2-\frac{s}{2},\frac{3+s}{2};-A^{-2}\right) \\ &\phantom{=}\;+\gamma_2(s)A^{1/2}F(s,A)+ \gamma_2(1-s)A^{1/2}F(1-s,A),  
\end{align*} where
\begin{align*}\gamma_1(s)&%=\frac{\Gamma(s/2-1) \Gamma((1-s)/2-1) \Gamma(3/2) \Gamma(1/2)}{\Gamma(s/2) \Gamma((1-s)/2) \Gamma(1/2) \Gamma(-1/2)}
=\frac{-1}{(s-2)(s+1)},
\\
\gamma_2(s)&%=\frac{ \Gamma(1-s/2) \Gamma(1/2-s) \Gamma(3/2)\Gamma(1/2)}{\Gamma(1) \Gamma((1-s)/2) \Gamma((3-s)/2) \Gamma((1-s)/2)} 
=  \frac{ \pi \Gamma(1-s/2) \Gamma(1/2-s) 
}{(1-s)  (\Gamma((1-s)/2))^3},  
\end{align*}
and $$F(s,u)=u^{1/2-s} \cdot {}_2F_1\left(\frac{s+1}{2},\frac{s-1}{2};s+\frac{1}{2};-u^{-2}\right).$$
Using the series representation of the hypergeometric function%(see equation (\ref{hgseriesdef}))
, we deduce that \begin{equation} \label{fsexp1}
F_s(A^2+1) = \gamma_2(s)A^{1/2}F(s,A)+ \gamma_2(1-s)A^{1/2}F(1-s,A)+O(t^{-2}). 
\end{equation}
We will now use this expansion to deal with the contribution of the small eigenvalues. In particular, we prove the following lemma.
\begin{lemma} For $f_\epsilon$ as in equation (\ref{feps}), we have that
    $$2\sum_{s_j \in (1/2,1)}d_{t_j}\left( f_{\epsilon} \right)\left|\hat{u}_{j}\right|^2=M_{\epsilon}(X,l)+O(1),$$
    and
$$2\sum_{s_j=1/2}d_{t_j}\left( f_{\epsilon} \right)\left|\hat{u}_{j}\right|^2=a_{1/2,\epsilon}(X,l)+O(1).$$
Hence,
 \begin{equation} \label{eepsmain} e_{\epsilon}(R,l)=2\sum_{t_j \in \mathbb{R}-\left\{0\right\}}d_{t_j}(f_{\epsilon})|\hat{u}_{j}|^2+O(1).\end{equation}
\end{lemma}
\begin{proof}
For $s \neq 0,1,1/2$ fixed real, by the series definition of the hypergeometric function (see \cite[\S 9.14, p.1010]{toisap}) , we have
$$\frac{F_s(A^2+1)}{A^{1/2}}=\gamma_2(s)A^{1/2-s}+\gamma_2(1-s)A^{s-1/2}+O(A^{-{1/2}}).$$
Hence,
\begin{align*}&d_t(f_{\epsilon})  \\
&= \frac{1}{\pi}\int_{R-\epsilon}^{R+\epsilon}\left( \gamma_2(s)(\sinh{x})^{1/2-s}+\gamma_2(1-s)(\sinh{x})^{s-1/2}+O((\sinh{x})^{-1/2})\right)\psi_{\epsilon}(R-x) dx
%\\ &=& \frac{1}{\pi}\int_{R-\epsilon}^{R+\epsilon}\left( \gamma_2(s)(\sinh{x})^{1/2-s}+\gamma_2(1-s)(\sinh{x})^{s-1/2}\right)\psi_{\epsilon}(R-x) dx+ O\left((\sinh{R})^{-3/2} \int_{R-\epsilon}^{R+\epsilon}\psi_{\epsilon}(R-x) dx\right)
\\ &= \frac{1}{\pi}\int_{R-\epsilon}^{R+\epsilon}\left( \gamma_2(s)(\sinh{x})^{1/2-s}+\gamma_2(1-s)(\sinh{x})^{s-1/2}\right)\psi_{\epsilon}(R-x) dx + O\left((\sinh{R})^{-1/2} \right).
\end{align*}
Therefore, the contribution of $s \neq 1/2$ real in $N_{\epsilon}(X,l)$ is $M_{\epsilon}(X,l)+O(1)$. %Alternatively, this also follows from \cite[Lemma 3.4, Proposition 3.5(i)]{lekkas}.
The second part follows directly by substituting equation (\ref{fsdef}) into (\ref{defaheps}), which is the definition of $a_{1/2,\epsilon}(X,l)$.
\end{proof}
Using equation (\ref{eepsmain}), it is sufficient to show that $$\sum_{t_j>0}d_{t_j}(f_{\epsilon})|\hat{u}_{j}|^2=\Omega_{\delta}((\log{R})^{1/4-\delta}).$$
To that end, we prove estimates for the corresponding hypergeometric functions. Combining the quadratic transformation given in \cite[15.8.18]{dlmf} with Euler's transformation \cite[15.8.1]{dlmf}, we arrive at the following identity:
%\begin{equation} \label{hypgoas}
%{}_2 F_1 \left(
%r+a,r+b; 2r+a+b+\frac{1}{2}
%; z 
%\right)=
%    \left(\frac{1+\sqrt{1-z} } {2} \right)^{1/2-a-b-2r} \cdot \left(1+ O(z\cdot r^{-1})\right) ,
%\end{equation}
$${}_2F_1\left(\frac{s+1}{2},\frac{s-1}{2};s+\frac{1}{2};z\right)=\left(\frac{1+\sqrt{1-z}}{2}\right)^{\frac{1}{2}-s}{}_2F_1\left(\frac{-1}{2},\frac{3}{2};s+\frac{1}{2};\frac{z}{2\left(1+\sqrt{1-z}\right)}\right),$$
where $|z|<1$ and complex powers are defined using a branch cut at the negative real axis.
Using the series expansion of the hypergeometric function, we deduce that 
$${}_2F_1\left(\frac{s+1}{2},\frac{s-1}{2};s+\frac{1}{2};z\right)=\left(\frac{1+\sqrt{1-z}}{2}\right)^{\frac{1}{2}-s}+O(z \cdot s^{-1}),$$
where the square root is taken to have positive real part. Hence, for $s=1/2+it$, we have \begin{align*}F(s,A)&=A^{-it}\left(\left(\frac{1+\sqrt{1+A^{-2}}}{2}\right)^{-it}+O(A^{-2}s^{-1})\right)\\ &= 2^{it}\left(A+\sqrt{A^2+1}\right)^{-it}+O(A^{-2}s^{-1}). 
\end{align*}
Therefore, via equation (\ref{fsexp1}), we have
\begin{align*}&A^{-1/2} \cdot F_s(A^2+1) \\
&=  O(A^{-1/2}t^{-2})+G(s)B^{-it}(1+O(A^{-2}t^{-1}))+G(1-s)B^{it}(1+O(A^{-2}t^{-1}))  \\
&= G(s) B^{-it}+G(1-s)B^{it}+O(A^{-1/2}t^{-2}+A^{-2}t^{-1}G(s)) 
\\
&= 2 \cdot \hbox{Re} \left(G(s) B^{-it}\right)+O(A^{-1/2}t^{-2}+A^{-2}t^{-1}G(s)), 
\end{align*}
where
$$B:=A+\sqrt{A^2+1},$$
and \begin{equation} \label{gdef} G(s):=2^{it}\gamma_2(s)=2^{it} \cdot  \frac{ \pi \Gamma(1-s/2) \Gamma(1/2-s) 
}{(1-s)  (\Gamma((1-s)/2))^3}. \end{equation}
By Stirling's formula, we have
$$|G(s)| \asymp t^{-1/2}.$$
Therefore,
by taking $A=\sinh{x}$, we arrive at
$$\left(\sinh{x}\right)^{-1/2}F_s(\cosh^2{x})=2 \cdot \hbox{Re}\left( G(s) e^{-itx} \right)+O(t^{-3/2}),$$
and, hence,
\begin{align*}d_t(f_{\epsilon})&=\frac{1}{\pi}\int_{-\infty}^{+\infty}\frac{1}{\sqrt{\sinh{x}}}\psi_{\epsilon}(R-x)F_s(\cosh^2{x})dx\\&=\frac{2}{\pi} \cdot \hbox{Re}(G(s)e^{-iRt}) \cdot \hat{\psi}_{\epsilon}(t)+O(t^{-3/2}). \end{align*}
We now use Lemma \ref{phole} to resonate several of the exponentials $e^{-iRt}$. For fixed $M,T>0$ to be chosen later, let $\mathcal{R}(T,M)$ be the set of all $R$ satisfying Lemma \ref{phole}, for $r_1,...,r_n$ equal to all values $t_j$ smaller than $M$.
For $R \in \mathcal{R}(T,M) $, we have
\begin{align}&\sum_{0 \leq t_j \leq M}\frac{2}{\pi} \cdot \hbox{Re}(G(s)e^{-iRt_j}) \cdot \hat{\psi}_{\epsilon}(t)|\hat{u}_{j}|^2 \nonumber \\
&=\sum_{0 \leq t_j \leq M}\frac{2}{\pi} \cdot \hbox{Re}(G(s)) \cdot \hat{\psi}_{\epsilon}(t)|\hat{u}_{j}|^2+O\left(T^{-1}\sum_{1 \leq t_j \leq M} \frac{|\hat{u}_{j}|^2}{t_j^{1/2}}\right) \nonumber \\
&= \sum_{0 \leq t_j \leq M}\frac{2}{\pi} \cdot \hbox{Re}(G(s)) \cdot \hat{\psi}_{\epsilon}(t)|\hat{u}_{j}|^2+O\left(T^{-1}M^{1/2}\right), \label{afterph} \end{align}
where, for the last equality, we used the first part of Corollary \ref{weightedperiods}.
To deal with the eigenvalues satisfying $t_j>M$, we use equation (\ref{fourierbound}) and the second part of Corollary \ref{weightedperiods} to write 
$$\sum_{t_j>M}d_{t_j}(f_{\epsilon})|\hat{u}_{j}|^2=O_k\left(\sum_{t_j>M}\epsilon^{-k}t_j^{-k-1/2}|\hat{u}_{j}|^2\right)=O_k\left(\epsilon^{-k}M^{1/2-k} \right),$$
where $k$ is some fixed constant larger than $1/2$.
Combining with equation (\ref{afterph}), we arrive at
$$\sum_{t_j>0}d_{t_j}(f_{\epsilon})|\hat{u}_{j}|^2=\!\! \! \sum_{0 \leq t_j \leq M}\frac{2}{\pi} \cdot \hbox{Re}(G(s)) \hat{\psi}_{\epsilon}(t)|\hat{u}_{j}|^2+O_k\left(T^{-1}M^{1/2}+\epsilon^{-k}M^{1/2-k}\right).$$
We optimize by taking $T=M^{1/2}$ and $M=\epsilon^{k/(1/2-k)}$, so that
\begin{equation} \label{ok1} \sum_{t_j>0}d_{t_j}(f_{\epsilon})|\hat{u}_{j}|^2=\sum_{0 \leq t_j \leq M}\frac{2}{\pi} \cdot \hbox{Re}(G(s)) \cdot \hat{\psi}_{\epsilon}(t)|\hat{u}_{j}|^2+O_k\left(1\right). \end{equation}
We now resonate the terms of the sum in the right hand side. To that end, we prove the following lemma.
\begin{lemma} \label{signs}
    For $s=1/2+it$ and $G(s)$ as above, we have that $$\left|\hbox{Re}\left(G(s)\right)\right| \asymp t^{-1/2}.$$
    Furthermore, there is some real constant $C>0$ such that, for any $t>C$,  $$\hbox{Re}\left(G(s)\right)>0.$$
\end{lemma}
\begin{proof}
It is enough to show that, as $t \rightarrow \infty$, the expression  $t^{1/2}G(s)$ converges to some complex constant with positive real part.

By applying Stirling's formula on equation \ref{gdef}, we have

\begin{align*}t^{1/2}G(s) &\sim 2^{it} \cdot \left(\frac{\pi}{2}\right)^{1/2}  \cdot \frac{t^{1/2}\left(1-s/2\right)^{(1-s)/2}\left(1/2-s\right)^{-s}}{(1-s)\left((1-s)/2\right)^{-3s/2}}  %\cdot e^{-(1-s/2)-(1/2-s)+3(1-s)/2} 
\\
&= 2^{it} \cdot \left(\frac{\pi}{2}\right)^{1/2}  \cdot \frac{t^{1/2}\left(3/4-it/2\right)^{1/4-it/2}\left(-it\right)^{-1/2-it}}{(1/2-it)\left(1/4-it/2\right)^{-3/4-3it/2}}  %\cdot e^{-(1-s/2)-(1/2-s)+3(1-s)/2} 
\\
& = \left(\frac{\pi}{8}\right)^{1/2} (-i)^{-1/2} \cdot \left(\frac{3/2-it}{1/2-it}\right)^{1/4} \cdot \left(\frac{\left(3/2-it\right)\left(-it\right)^{2}}{\left(1/2-it\right)^{3}}\right)^{-it/2}
\\
& = \left(\frac{\pi}{8}\right)^{1/2} (-i)^{-1/2} \cdot \left(1+O(t^{-1})\right)^{1/4} \cdot \left(1+O(t^{-2})\right)^{-it/2}.
\end{align*}
Hence, $$t^{1/2}G(s) \rightarrow \left(\frac{\pi}{8}\right)^{1/2} (-i)^{-1/2}=\left(\frac{\pi}{8}\right)^{1/2} \cdot \frac{1+i}{\sqrt{2}},$$
which has positive real part, as required.
\end{proof}
Take $\tau>0$ such that, for $|t|<\tau$, we have $\hat{\psi}(t)>1/2$. By the definition of $\psi_\epsilon$, this implies that, for $t \leq \tau/\epsilon$, we have $\hat{\psi}_{\epsilon}(t)>1/2$. Note also that
$$M=\epsilon^{k/(1/2-k)}=(1/\epsilon)^{1+1/(2k-1)},$$
and so, for small $\epsilon$, we have $M > \tau/\epsilon$.
Hence, Lemma \ref{signs} gives $$\left|\sum_{C \leq t_j \leq M}\frac{2}{\pi} \cdot \hbox{Re}(G(s)) \cdot \hat{\psi}_{\epsilon}(t)|\hat{u}_{j}|^2\right| \gg \sum_{C \leq t_j \leq \tau/\epsilon}t^{-1/2}\cdot |\hat{u}_{j}|^2 \gg (\tau/\epsilon)^{1/2},$$
where the second bound follows from equation Lemma \ref{periodslemma} and partial summation.
%On the other hand,
%$$\sum_{0 \leq t_j \leq C}\frac{2}{\pi} \cdot \hbox{Re}(G(s)) \cdot \hat{\psi}_{\epsilon}(t)|\hat{u}_{j}|^2 \ll \sum_{0 \leq t_j \leq C}t^{-1/2}\cdot |\hat{u}_{j}|^2 \ll 1. $$
We conclude that
$$\left|\sum_{0 \leq t_j \leq M}\frac{2}{\pi} \cdot \hbox{Re}(G(s)) \cdot \hat{\psi}_{\epsilon}(t)|\hat{u}_{j}|^2\right| \gg (\tau/\epsilon)^{1/2}+O(1). $$
Therefore, equation (\ref{ok1}) gives
$$\left|\sum_{t_j>0}d_{t_j}(f_{\epsilon})|\hat{u}_{j}|^2\right| \gg (\tau/\epsilon)^{1/2}+O_k\left(1\right).$$
 Recall now that, as $R \in \mathcal{R}(T,M)$, we must have $M<R<MT^n$, where $n$ is the number of values of $t_j$ smaller than $M$. As $T=M^{1/2}$, it follows that
  $$\log{M} < \log{R} < \log{M} + (n/2)  \cdot \log{M}. $$
 On the other hand, by \cite[Eq.(7.11)]{iwaniec}, we have that $n \ll M^2$. Hence,
  $$M^2 \gg n \gg  \log{R}/\log{M} \gg \log{R}/\log{\log{R}}. $$
Therefore, for $R \in \mathcal{R}(T,M)$ and for any fixed $\delta>0$, we have
$$\epsilon^{-1/2} = M^{(2k-1)/(4k)} \gg (\log{R})^{1/4-\delta},$$
from which we conclude that
$$\sum_{t_j>0}d_{t_j}(f_{\epsilon})|\hat{u}_{j}|^2 
  \gg (\log{R})^{1/4-\delta}.$$
  By taking $\epsilon \rightarrow 0$, and therefore $M \rightarrow +\infty$, the elements of $\mathcal{R}(T,M)$ become arbitarily large. Therefore, we have that
  $$\sum_{t_j>0}d_{t_j}(f_{\epsilon})|\hat{u}_{j}|^2 
  = \Omega_{\delta}\left( (\log{R})^{1/4-\delta}\right),$$
as required.

\section{Acknowledgements}
The author would like to thank his PhD advisor, Yiannis Petridis, for suggesting the problem, and for his advice and patience. The author was supported by University College London and the Engineering and Physical Sciences Research Council (EPSRC) studentship grant EP/V520263/1. The author was also supported by the Swedish Research Council under grant no.2016-06596 while in residence at Institut Mittag-Leffler in Djursholm, Sweden during the Analytic Number Theory program in the spring of 2024. 
Finally, the author
would like to thank the Max Planck Institute for Mathematics in Bonn, for its hospitality
and financial support.
\addcontentsline{toc}{chapter}{Bibliography}

% Actually generates your bibliography.
%\bibliography{example}

\begin{thebibliography}{99}
\bibitem{chamizo2} F. Chamizo. \emph{Some applications of large sieve in Riemann surfaces}, Acta Arith., 77, no. 4, pp. 315–337, 1996.

%\bibitem{bateman} H. Bateman. \emph{Tables of Integral Transforms}, Bateman Manuscript Project, McGraw-Hill Book Company, New York, 1954.


%\bibitem{chamizoth} F. Chamizo. \emph{Topics in Analytic Number Theory}, Doctoral Thesis, Universidad Aut\'onoma de Madrid, 1994.

%\bibitem{chamizok} F. Chamizo, D. Raboso. \emph{On the Kuznetsov formula}, J. Funct. Anal., 268, no. 4, pp. 869–886, 2015.

\bibitem{chatzakosomega} D. Chatzakos. \emph{$\Omega$-results for the hyperbolic lattice point problem}, Proc. Amer. Math. Soc., 145, no. 4, pp. 1421–1437, 2017.

\bibitem{chatzakosomega2} D. Chatzakos. \emph{Mean value and $\Omega$-results for the hyperbolic lattice point problem in conjugacy classes}, Rev. Mat. Iberoam., 35, no. 4, pp. 1123–1152, 2019.

%\bibitem{chatzakos} D. Chatzakos, Y. Petridis. \emph{The hyperbolic lattice point problem in conjugacy classes}, Forum Math., 28, pp. 981–1003, 2016.

%\bibitem{chatzakosequid} D. Chatzakos, P. Kurlberg, S. Lester, I. Wigman. \emph{On the distribution of lattice points on hyperbolic circles}, Algebra Number Theory, 15, no. 9, pp. 2357–2380, 2021.

%\bibitem{petru} P. Constantinescu. \emph{Dissipation of correlations of holomorphic cusp forms}, arXiv:2112.01427, 2021.

\bibitem{good} A. Good. \emph{Local analysis of Selberg's trace formula}, Lecture Notes in Mathematics, 1040, Springer-Verlag, Berlin, i+128 pp., 1983.

\bibitem{toisap} I. S. Gradshteyn, I. M. Ryzhik. \emph{Table of integrals, series, and products}, translated from the Russian. Translation edited and with a preface by A. Jeffrey and D. Zwillinger, 7th ed., Elsevier/Academic Press, Amsterdam, xlviii+1171 pp., 2007.

%\bibitem{gunther} P. G\"{u}nther. \emph{Gitterpunktprobleme in symmetrischen Riemannschen R\"aumen vom Rang 1}, Math. Nachr., 94, pp. 5–27, 1980.


 %\bibitem{fay} J. D. Fay. \emph{Fourier coefficients of the resolvent for a Fuchsian group}, J. Reine Angew. Math., 293/294, pp. 143–203, 1977.



\bibitem{hejhal} D. A. Hejhal. \emph{Sur certaines s{\'e}ries de Dirichlet dont les p{\^o}les sont sur les lignes critiques}, C. R. Acad. Sci. Paris S{\'e}r. A, 287, pp. 383–385, 1978.


\bibitem{huber} H. Huber. \emph{Ein Gitterpunktproblem in der hyperbolischen Ebene}, J. Reine Angew. Math., 496, pp. 15–53, 1998.


\bibitem{iwaniec} H. Iwaniec. \emph{Spectral Methods of Automorphic Forms}, 2nd ed., Graduate Studies in Mathematics, 53, American Mathematical Society, Providence, RI; Revista Matemática Iberoamericana, Madrid, xii+220 pp., 2002.


%\bibitem{katok} S. Katok. \emph{Fuchsian Groups}, Chicago Lectures in Mathematics, University of Chicago Press, Chicago, IL, pp. x+175, 1992.
\bibitem{voskoulekkas} D. Lekkas, M. Voskou. \emph{Large sieve inequalities for periods of Maass forms}, Res. Number Theory, 10, no. 3, Paper No. 73, 19 pp., 2024.

\bibitem{lekkas} D. Lekkas. \emph{A relative trace formula and counting geodesic segments in the hyperbolic plane}, Doctoral Thesis (Ph.D), UCL (University College London), 2023.

%\bibitem{voskoulekkas} D. Lekkas, M. Voskou. \emph{Large sieve inequalities for periods of Maass forms}, Res. Number Theory, 10, no. 3, Paper No. 73, 19 pp., 2024.
%\bibitem{luke} Y. L. Luke. \emph{The special functions and their approximations, Vol. I}, Mathematics in Science and Engineering, 53, Academic Press, New York-London, xx+349 pp., 1969.


\bibitem{mckee} K. Martin, M. McKee, E. Wambach. \emph{A relative trace formula for a compact Riemann surface}, Int. J. Number Theory, 7, no. 2, pp. 389–429, 2011.

%\bibitem{nordentoft} A. C. Nordentoft, Y. N. Petridis, M. S. Risager. \emph{Bounds on shifted convolution sums for Hecke eigenforms}, Res. Number Theory, 8, no. 2, Paper No. 26, 20 pp., 2022.

\bibitem{dlmf} F. W. J. Olver, A. B. Olde Daalhuis, D. W. Lozier, B. I. Schneider, R. F. Boisvert, C. W. Clark, B. R. Miller, B. V. Saunders, H. S. Cohl, M. A. McClain, Eds. \emph{NIST Digital Library of Mathematical Functions}, http://dlmf.nist.gov/, Release 1.1.7, 2022-10-15.

 

\bibitem{parkkonen2} J. Parkkonen, F. Paulin. \emph{Counting common perpendicular arcs in negative curvature}, Ergodic Theory Dynam. Syst., 37, no. 3, pp. 900–938, 2017.

\bibitem{rudnick} R. Phillips, Z. Rudnick. \emph{The circle problem in the hyperbolic plane}, J. Funct. Anal., 121, no. 1, pp. 78–116, 1994.

    \bibitem{selberg}A. Selberg. 
 \emph{Equidistribution in discrete groups and the spectral theory of
automorphic forms.}
\url{http://publications.ias.edu/selberg/section/2491.}
\bibitem{tsuzukiletter} M. Tsuzuki. \emph{Letter to Kimball Martin regarding the paper `A relative trace
formula for a compact Riemann surface' .} Unpublished.


\bibitem{tsuzuki} M. Tsuzuki. \emph{Spectral means for period integrals of wave functions on real hyperbolic spaces}, J. Number Theory, 129, pp. 1387–2438, 2009.

\bibitem{voskou} M. Voskou. \emph{Refined Counting of Geodesic Segments in the Hyperbolic Plane}, preprint, 2024. Available: \url{https://doi.org/10.48550/arXiv.2407.03134}.


% All done. \o/
\end{thebibliography}

\end{document}